\documentclass[12pt, reqno]{amsart}
 \usepackage{amsmath, amsthm, amscd, amsfonts, amssymb, graphicx, color, float}
\usepackage[bookmarksnumbered, colorlinks, plainpages]{hyperref}

\newtheorem{theorem}{Theorem}[section]
\newtheorem{lemma}[theorem]{Lemma}
\newtheorem{proposition}[theorem]{Proposition}
\newtheorem{corollary}[theorem]{Corollary}
\theoremstyle{definition}
\newtheorem{definition}[theorem]{Definition}
\newtheorem{example}[theorem]{Example}

\theoremstyle{remark}

\numberwithin{equation}{section}
\begin{document}

\title[Frames for operators in Banach spaces]
{Frames for operators in Banach spaces \\via semi-inner products}

\author[B. Dastourian]{Bahram Dastourian$^*$}
\address{Bahram Dastourian, Department of Pure Mathematics, Ferdowsi University of Mashhad, Mashhad, P.O. Box 1159-91775, Iran}
\email{bdastorian@gmail.com}

\author[M. Janfada]{Mohammad Janfada}
\address{Mohammad Janfada, Department of Pure Mathematics, Ferdowsi University of Mashhad, Mashhad, P.O. Box 1159-91775, Iran}
\email{mjanfada@gmail.com}

\begin{abstract}
In this paper, we propose to define the concept of family of local atoms and then we generalize this concept to the atomic system for operator in Banach spaces by using semi-inner product. We also give a characterization of atomic systems leading to obtain new frames. In addition, a reconstruction formula is obtain. Next, some new results are established. The characterization of atomic systems  allows us to state some results  for sampling theory in semi-inner product reproducing kernel Banach spaces. Finally, having used frame operator in Banach spaces, new perturbation results are established.
\end{abstract}

\maketitle

\section{Introduction}
Atomic systems and $K$-frames, where $K$ is a bounded linear operator on separable Hilbert space $\mathcal H$, introduced by L. G\u{a}vru\c{t}a in \cite{Guv} as a generalization of families of local atoms \cite{FW}. A sequence $\{ f_j\}_{j\in \mathbb{N}}$ in the Hilbert space $\mathcal H$ is called an \textit{atomic system} for the bounded linear operator $K$ on $\mathcal H$ if\\
(i) the series $\sum_{j\in \mathbb{N}}c_j x_j$ converges for all $c = (c_j) \in l^2:=\{\{b_j\}_{j\in \mathbb{N}} : \sum_{j\in \mathbb N}|b_j|<\infty\}$;\\
(ii) there exists $C > 0$ such that for every $f\in \mathcal H$ there exists $a_f=(a_j)\in l^2$ such that $\|a_f\|_{l^2}\leq \|f\|$ and $Kf =\sum_{j\in \mathbb{N}}a_j f_j$.\\
Also a sequence $\{ f_j\}_{j\in \mathbb{N}}$ is said to be a \textit{$K$-frame} for $\mathcal H$ if there exist constants $A, B > 0 $ such that
\[ A\|K^*f\|^2\leq\sum_{j\in \mathbb{N}}|\langle f, f_j\rangle|^2\leq B\|f\|^2,~~ f\in \mathcal H.\]
It is proved that these two concepts are equivalent \cite{Guv}. We refer to \cite{XZG} for more results on these concepts. In addition, we generalized these two concepts and gave some new results \cite{DJ}. Note that frames in Hilbert spaces are just a particular case of $K$-frames, when $K$ is the identity operator on these Hilbert spaces. Frames in Hilbert spaces were introduced by Duffin and Schaffer \cite{Duf} in 1952. In 1986, frames were brought to life by Daubechies \textit{et al.} \cite{Daub}. Now frames play an important role not only in the theoretics but also in many kinds of applications, and have been widely applied in signal processing \cite{Fer}, sampling \cite{Eld1, Eld2}, coding and communications \cite{Stro}, filter bank theory \cite{Bol}, system modeling \cite{Dud}, and so on.

However, the theoretical research of frames for Banach spaces is quite different from that of Hilbert spaces. Due to the lack of an inner-product, frames for
Banach spaces were simply defined as a sequence of linear functionals from $X^*$, the dual space of $X$, rather than a sequence of basis-like elements in
$X$ itself. Feichtinger and Gr\"{o}cheing \cite{FG} extended the notion of atomic decomposition to Banach space. Gr\"{o}cheing \cite{G} introduced a more
general concept for Banach spaces called Banach frame. Aldroubi \textit{et al.} \cite{AST} introduced $p$-frames and then this frame was discussed in separable Banach
space in \cite{CS}. Now we are going to state frames for separable Banach spaces by Casazza \textit{et al.} \cite{CCS}. In so doing, they introduced the following definition.
\begin{definition}
A sequence space $X_d$ is called a \textit{BK-space}, if it is a Banach space and the coordinate functionals are continuous on $X_d$, i.e. the relations
$x_n=\{\alpha_j^{(n)}\}, x=\{\alpha_j\}\in X_d, \lim_{n\to\infty}x_n=x$ imply $\lim_{n\to\infty}\alpha_j^{(n)}=\alpha_j (j=1, 2, 3, \cdots).$
\end{definition}
Also we add some notions about this Banach space. If the canonical vectors form a Schauder basis for $X_d$, then $X_d$ is called a \textit{CB-space} and its canonical basis is denoted by $\{e_j\}$. If $X_d$ is reflexive and a CB-space, then $X_d$ is called an \textit{RCB-space}.

Frame for separable Banach space \cite{CCS} is introduced as follows.
\begin{definition}
Let X be a separable Banach space and $X_d$ be a BK-space. A countable family $\{f_j\}$ in the dual $X^*$ is called an \textit{$X_d$-frame} for $X$ if \\
$(i)$ $\{f_j(f)\}\in X_d, \forall f\in X$;\\
$(ii)$ the norm $\|f\|_X$ and $\|\{f_j(f)\}\|_{X_d}$ are equivalent, i.e. there exist constants $A,B>0$ such that
\[ A\|f\|_X\leq \|\{f_j(f)\}\|_{X_d}\leq B\|f\|_X, ~~~ \forall f\in X.\]
When $A$ and $B$ are called $X_d$-frame bounds. $\{f_j\}$ is called an \textit{$X_d$-Bessel sequence} for $X$, if at least $(i)$ and the upper frame condition are satisfied.
\end{definition}
In 2011, H. Zhang and J. Zhang \cite{ZZ} introduced a new definition of $X_d$-frames by using semi-inner product. Under such a definition, an $X_d$-frame is
exactly a sequence of elements in $X$ itself.

 The main purpose of this paper is to provide a language for the study of families of local $X_d^*$-atoms, $X_d^*$-atomic systems and $X_d^*$-$K$-frames in Banach spaces via a compatible semi-inner product, which is a natural substitute for inner products on Hilbert spaces. We obtain some new results. In particular, we characterize $X_d^*$-$K$-frames. As a consequence, we state the main result for any semi-inner product reproducing kernel Banach spaces. Our last result of this paper is to show that the Casazza--Christensens perturbation theorem \cite{CC} of Hilbert space frames somehow holds for any $X_d^*$-$K$-frames by making use $X_d^*$-$K$-frame operator.

This paper is organized as follows. In Section 2 we will give the notation used in the paper, especially the definition of semi-inner product and
its properties. In Section 3 we extend notions about families of local $X_d^*$-atoms, $X_d^*$-atomic systems and $X_d^*$-$K$-frames in Banach spaces and new results are given. In section 4 by applying results from Section 3, we will present new result in reproducing kernel Banach spaces. Finally, in the last section we define $X_d^*$-$K$-frame operator and state its properties and then a perturbation of $X_d^*$-$K$-frame, the so-called Paley-Wiener perturbation \cite{P1, P2, P3}, is given by use of $X_d^*$-$K$-frame operator.
\section{Preliminaries}
We first state the following concept introduced by Lumer \cite{L} in 1961 but its main properties discovered by Giles \cite{Gil}, Nath \cite{N} and others.

A \textit{semi-inner product} (in short s.i.p.) on a complex vector space $X$ is a complex valued function $[f, g]$ on $X\times X$ with the following properties:\\
$1. [\lambda f+g, h]=\lambda[f, h]+[g, h]$ and $[f, \lambda g]=\overline{\lambda}[f, g]$, ~~~for all complex $\lambda$,\\
$2. [f, f]\geq 0$, ~~~ for all $f\in X$ and $[f, f]=0$ implies $f=0$,\\
$3. |[f, g]|^2\leq [f,f] [g,g].$\\

It was shown in \cite{L} that if $[., .]$ is a s.i.p. on $X$ then $\|f\|:=[f, f]^{\frac{1}{2}}$ is a norm on $X$, and in this situation, the semi-inner product is called compatible. Conversely, if $X$ is a normed vector space then it has a s.i.p. that induces its norm in this manner. A compatible semi-inner product is non-additive with respect to its second variable. Namely,
\[ [f, g+h]\neq[f, g]+[f, h], \quad f, g\in X, \]
in general. The concept of s.i.p. has been proved useful both theoretically and practically, and has been widely applied in the theory of functional analysis \cite{F, L, PP, P}, machine learning on reproducing kernel Banach spaces (in short RKBS) \cite{ZX} and so on. The reader is referred to \cite{D} for more information about semi-inner products. By properties 2 and 3 of semi-inner products, for each $f\in X$ the function that sends $g\in X$ to $[g, f]$ is a bounded linear functional on $X$. We shall denote this functional associated with $f$ by $f^*$ and call it the dual element of $f$. In other words, $f^*$ lies in the dual space $X^*$ of $X$. The mapping $f \to f^*$ will be called the duality mapping from $X$ to $X^*$.

A Banach space $X$ is called \textit{strictly convex}, whenever $\|f+g\|_X=\|f\|_X+\|g\|_X$ where $f, g\neq 0$ then $f=\alpha g$ for some $\alpha>0$. In this space the duality mapping from $X$ to $X^*$ is bijective \cite{F}. In other words, for each $f^*\in X^*$ there exists a unique $g\in X$ such that
$f^*(g)=[g, f]$ for all $g\in X$.  Moreover, we have $\|f^*\|_{X^*}=\|f\|_X$. Also in this case, $[f^*, g^*]_*:=[g, f], ~~~ f, g\in X$ defines a compatible semi-inner product on $X^*$ \cite{Gil}.

A Banach space $X$ is uniformly convex if for all $\epsilon>0$ there exists a $\delta>0$ such that $\|f+g\|_X\leq 2-\delta$ for all $f, g\in X$ with $\|f\|_X=\|g\|_X=1$ and $\|f-g\|_X> \epsilon.$
Bear in mind that a uniformly convex Banach space is reflexive \cite[page 134]{Con} but a reflexive Banach space is not necessarily uniformly convex \cite{MMD}. Every uniformly convex Banach space is automatically strictly convex.

H. Zhang and J. Zhang \cite{ZZ} introduce $X_d$-frame for Banach spaces via semi-inner products as follows.
\begin{definition}
Let $[., .]$ be a compatible s.i.p. on separable strictly convex s.i.p. Banach space $X$ and $X_d$ be an RCB-space. Then the family $\{f_j\}\subseteq X$ is called \textit{$X_d$-frame} for $X$ if\\
$(i)$ $\{[f, f_j]\}\in X_d, ~~~~~ \forall f\in X$;\\
$(ii)$ there exist constants $A,B>0$ such that
\[ A\|f\|_X\leq \|\{[f, f_j]\}\|_{X_d}\leq B\|f\|_X, ~~~~~ \forall f\in X.\]
\end{definition}
Recently, Zheng and Yang \cite{ZY} have presented $X_d$-frame for separable uniformly convex Banach spaces via semi-inner products, when $X_d$ is just a BK-space or CB-space.

Let $B(X, Y)$ be the bounded linear operator defined on Banach space $X$ with values in Banach space $Y$. We also write $B(X)$ instead of $B(X,X)$. We say that $T\in B(X, Y)$ majorizes $Q\in B(X, Z)$ if there exists $\gamma>0$ such that $\|Qf\|\leq \gamma\|Tf\|$ for all $f\in X$ \cite{B}. We also recall that a closed subspace $M$ of $X$ is complemented if there exists a closed subspace $N$ of $X$ such that $X$ can be written as the topological sum of $M$ and $N$. The range of any operator $T$ is denoted by $R(T)$. Finally, we cite some useful lemmas that will be used in the sequel as follow.
\begin{lemma}\cite{B}\label{maj} Suppose that $T\in B(X, Y)$, $U\in B(X, Z)$, and $V\in B(Z, Y)$. Then the following statements hold.\\
$(i)$ If $T$ majorizes $U$ and $\overline{R(T)}$ is complemented, then there exists $Q\in B(Y, Z)$ such that $U=QT$;\\
$(ii)$ Assume that $T$ majorizes $U$. Then $R(U^*)\subseteq R(T^*)$;\\
$(iii)$ Assume that $R(V)\subseteq R(T)$. Then $T^*$ majorizes $U^*$.
\end{lemma}
\begin{lemma}\cite[page 201]{KA}\label{star}
If $X_d$ is a CB-space with the canonical unit vectors $e_j$, $j\in J$ then the space $X_d^\circledast:=\{\{F(e_j)\}_{j\in J} : F\in X_d^*\}$ with the norm $\|\{F(e_j)\}_{j\in J}\|_{X_d^\circledast}:=\|F\|_{X_d^*}$  is a BK-space isometrically isomorphic to $X_d^*$. Also, every continuous linear functional $\Psi$ on $X_d$ has the form
\[ \Psi(\{c_j\})=\sum_jc_jd_j, \]
where $\{d_j\}\in X_d^\circledast$ is uniquely determined by $d_j=\Psi(e_j)$ and
\[ \|\Psi\|=\|\{\Psi(e_j)\}\|_{X_d^\circledast}. \]
\end{lemma}
When $X_d^*$ is a CB-space then its canonical basis is denoted by $\{e_j^*\}$.
\section{Families of local $X_d^*$-atoms, $X_d^*$-Atomic systems, and $X_d^*$-$K$-frames}
Throughout this section we assume that $X$ is a separable Banach space which is strictly convex and $[., .]$ is a compatible semi-inner product on $X$.

In this section we discuss families of local $X_d^*$-atoms, $X_d^*$-atomic systems, and $X_d^*$-$K$-frames in Banach spaces via semi-inner products.

First of all we give the following definition.
\begin{definition}
Let $K\in B(X)$ and $X_d$ be a BK-space. The family $\{f_j\}\subseteq X$ is called an \textit{$X_d^*$-$K$-frame} for $X^*$ if $\{[f_j, f]\}\in X_d^*$, for all $f$ in $X$ and there exist $A, B>0$ such that
\[ A\|K^*f^*\|_{X^*}\leq \|\{[f_j, f]\}\|_{X_d^*}\leq B\|f^*\|_{X^*}.\]
\end{definition}
The element $A$ is called the lower $X_d^*$-$K$-frame bound and $B$ is called the upper $X_d^*$-$K$-frame bound or just upper $X_d^*$-frame bound. If the right side of this inequality holds then we say that $\{f_j\}$ is an $X_d^*$-Bessel sequence for $X^*$. Especially, when $K=I$, the identity operator on $X$, then $\{f_j\}$ is called $X_d^*$-frame for $X^*$.

Now we present our definition of families of local $X_d^*$-atoms for Banach spaces. Actually we shall generalize the classical theory of families of local atoms for Hilbert spaces to Banach spaces  via semi-inner products.
\begin{definition}\label{def}
Let $\{f_j\}\subseteq X$ be an $X_d^*$-Bessel sequence for $X^*$ and let $X_0$ be a closed subspace of $X$. We call $\{f_j\}$ is a \textit{family of local $X_d^*$-atoms} for $X_0^*$ if there exists a linear functional $\{\mu_j\}\subseteq X_0^*$ such that\\
$(i)$ $\{\mu_j(f)\}\in X_d$ and there exists $C>0$ such that $\|\{\mu_j(f)\}\|_{X_d}\leq C\|f\|_X$,\\
$(ii)$ $f=\sum_j\mu_j(f)f_j$,\\
 for all $f\in X_0$.
\end{definition}
Now for the sequence $\{\mu_j\}\subseteq X_0$, by the Hahn-Banach theorem there exists $\{h_j\}\subseteq X^*$ for which $\|h_j\|=\|\mu_j\|$ and $h_j\mid_{X_0^*}=\mu_j$. But by the duality map from $X$ to $X^*$, $h_j=g_j^*$ for some $g_j\in X$. So from the condition $\mu_j(f)=[f, g_j]$, $(i)$ and $(ii)$ in Definition \ref{def} is equivalent to say that \\
$(i)'$ $\{g_j^*\}$ is an $X_d$-Bessel sequence for $X_0$,\\
$(ii)'$ $f=\sum_j[f, g_j]f_j$,\\
 for all $f\in X_0$.
\begin{proposition}
 Let $\{f_j\}\subseteq X$ be a family of local $X_d^*$-atoms for $X_0^*$. Then $\{f_j\}$ is an $X_d^*$-frame for $X_0^*$.
\end{proposition}
\begin{proof}
It is enough to show that $\|\{[f_j, f]\}\|_{X_d^*}\geq A\|f^*\|_{X_0^*}$, for some $A>0$. By definition of family of local $X_d^*$-atoms there exists a linear functional $\{g_j^*\}\subseteq X_0^*$ such that $f=\sum_j[f, g_j]f_j$ when $\{[f, g_j]\}\in X_d$ and $\|\{[f, g_j]\}\|_{X_d}\leq C\|f^*\|_{X_0^*}$ for some $C>0$ and for any $f^*$ in $X_0^*$. So
\begin{eqnarray}\label{1}
\|f\|^2&=&[f, f]=[\sum_j[f, g_j]f_j, f]=\sum_j[f, g_j][f_j, f]\nonumber\\
&\leq& \|\{[f, g_j]\}\|_{X_d}\|\{[f_j, f]\}\|_{X_d^*}\leq C\|f^*\| \|\{[f_j, f]\}\|_{X_d^*}.
\end{eqnarray}
This implies that $\|\{[f_j, f]\}\|_{X_d^*}\geq \frac{1}{C}\|f^*\|_{X_0^*}$. Thus $\{f_j\}$ is an $X_d^*$-frame for $X_0^*$ with lower $X_d^*$-frame bound $\frac{1}{C}$.
\end{proof}
To generalize the concept of a family of local $X_d^*$-atoms, we state the following lemma.
\begin{lemma}\label{bess}
Let $X_d$ be a CB-space and $\{f_j\}\subseteq X$. If $\sum_jb_jf_j$ converges in $X$, for all $b=\{b_j\}$ in $X_d$ and
$\|\sum_jb_jf_j\|_X\leq B\|b\|_{X_d}$, then $\{f_j\}$ is an $X_d^*$-Bessel sequence for $X^*$ with bound $B$.
\end{lemma}
\begin{proof}
Since $\sum_jb_jf_j$ converges in $X$ for every $b=\{b_j\}$ in $X_d$, we can define the operator $U:X_d\to X$ by $U(\{b_j\})=\sum_jb_jf_j$.
Then we have $\|U\|\leq B$, $U(e_j)=f_j$ and its dual operator is a bounded operator which is defined by $U^*:X^*\to X_d^*$ satisﬁes:
\[ U^*(f^*)(e_j)=f^*(U(e_j))=f^*(f_j)=[f_j, f].\]
So by Lemma \ref{star}, $\{U^*(f^*)(e_j)\}=\{[f_j, f]\}$ is identified with $U^*f^*$. Therefore, $\|\{[f_j, f]\}\|_{X_d^*}=\|U^*f^*\|\leq \|U\|~\|f\|\leq B~\|f\|$. It means, $\{f_j\}$ is an $X_d^*$-Bessel sequence for $X^*$ with bound $B$.
\end{proof}
The generalization of a family of local $X_d^*$-atoms is given below.
\begin{definition}
Let $\{f_j\}\subseteq X$ and $K\in B(X)$. $\{f_j\}$ is called an \textit{$X_d^*$-atomic system} for $X^*$ with respect to $K$ if\\
$(i)$ $\sum_jb_jf_j$ converges in $X$ for all $b=\{b_j\}$ in $X_d$ and there exists $B>0$ such that $\|\sum_jb_jf_j\|_X\leq B\|b\|_{X_d}$;\\
$(ii)$ There exists $C>0$ such that for every $f\in X$ there exists $a_f=\{a_j\}\in X_d$ such that $\|a_f\|\leq C\|f\|$ and $Kf=\sum_ja_jf_j$.
\end{definition}
Indeed, every family of local $X_d^*$-atoms is an $X_d^*$-atomic system. Note that the part $(i)$ says that $\{f_j\}$ is an $X^*_d$-Bessel sequence for $X^*$ by Lemma \ref{bess}.
One of the most important results of this paper is given below. Actually, this is a relation between $X_d^*$-atomic systems and $X^*_d$-$K$-frames. Besides, a new reconstruction is obtained.
\begin{theorem}
Let $X_d$ be a CB-space and $\{f_j\}$ be an $X^*_d$-Bessel sequence for $X^*$. Suppose that $T:X^* \to X_d^*$ is given by $T(f^*)=\{[f_j, f]\}$ and $\overline{R(T)}$ is complemented in $X_d^*$. Then the following statements are equivalent:\\
$(i)$ $\{f_j\}$ is an $X_d^*$-atomic system for $X^*$ with respect to $K$;\\
$(ii)$ $\{f_j\}$ is an $X_d^*$-$K$-frame for $X^*$;\\
$(iii)$ $\{f_j\}$ is an $X_d^*$-Bessel sequence for $X^*$ and there exists an $X_d$-Bessel sequence $\{g_j^*\}$ for $X$ such that for any $f^*\in X^*$, we have
\[ K^*f^*=\sum_j[f_j, f]g_j^*.\]
\end{theorem}
\begin{proof}
$(i)\Rightarrow (ii).$  For every $f^*\in X^*$, we have
\begin{eqnarray*}
\|K^*f^*\|&=&\sup_{g\in X, \|g\|=1}|K^*f^*(g)|\\
&=&\sup_{g\in X, \|g\|=1}|f^*(Kg)|=\sup_{g\in X, \|g\|=1}|[Kg, f]|,
\end{eqnarray*}
by definition of an $X_d^*$-atomic system there exists $\{b_j\}\in X_d$ such that $\|\{g_j\}\|\leq C\|g\|$, for some $C>0$, $Kg=\sum_jb_jf_j$, so
\begin{eqnarray*}
\|K^*f^*\|&=&\sup_{g\in X, \|g\|=1}|\sum_jb_j[f_j, f]|\\
&\leq& \sup_{g\in X, \|g\|=1}\|\{b_j\}\|_{X_d} \|\{[f_j, f]\}\|_{X_d^*}\leq C \|\{[f_j, f]\}\|_{X_d^*}.
\end{eqnarray*}
Therefore the lower $X_d^*$-$K$-frame bound of $\{f_j\}$ is $\frac{1}{C}$, i.e. $\frac{1}{C}\|K^*f^*\|\leq \|\{[f_j, f]\}\|_{X_d^*}$.

$(ii)\Rightarrow (iii).$ Since $T$ majorizes $K^*$ and $\overline{R(T)}$ is complemented in $X_d^*$, by Lemma \ref{maj} there exists a bounded operator $Q:X_d^*\to X^*$ such that $K^*=QT$. So
\[ K^*(f^*)=QT(f^*)=Q(\sum_j[f_j, f]e_j^*)=\sum_j[f_j, f]Qe_j^*,\]
for $Qe_j^*\in X^*$ there exists a unique $g_j\in X$ such that $g_j^*=Qe_j^*$. So by Lemma \ref{star} we have,
\[ \|\{[g_j^*, g^*]\}\|_{X_d}=\|\{g(g_j^*)\}\|=\|\{g(Q(e_j^*)\}\|=\|\{Q^*(g)(e_j^*)\}\|=\|Q^*(g)\|\leq \|Q\| \|g\|_{X}, \]
it means, $\{g_j^*\}$ is an  $X_d$-Bessel sequence for $X$.

$(iii)\Rightarrow (i).$ Suppose that $\{f_j\}$ is an $X_d^*$-Bessel sequence for $X^*$ with bound $B$ and $b=\{b_j\}\in X_d$. We estimate for positive integers $m>n$ that
\begin{eqnarray}\label{con}
\|\sum_{j\in J_m\backslash J_n}b_jf_j\|&=&\sup_{f^*\in X^*, \|f^*\|\leq1}|f^*(\sum_{j\in J_m\backslash J_n}b_jf_j)|\nonumber\\
&=&\sup_{f^*\in X^*, \|f^*\|\leq1}|\sum_{j\in J_m\backslash J_n}b_j[f_j, f]|\nonumber\\
&\leq& \|\sum_{j\in J_m\backslash J_n}b_je_j\|_{X_d}\sup_{f^*\in X^*, \|f^*\|\leq1}\|\{[f_j, f]\}\|_{X_d^*}\nonumber\\
&\leq& B \|\sum_{j\in J_m\backslash J_n}b_je_j\|_{X_d},
\end{eqnarray}
as $e_j$ form a Schauder basis for $X_d$, $\|\sum_{j\in J_m\backslash J_n}b_je_j\|_{X_d}$ goes to zero as $m, n$ tend to infinity. As a result $\sum_{j\in J}b_jf_j$ converges in $X$.

Using the same technique as that engaged in (\ref{con}), we obtain
\[ \|\sum_jb_jf_j\|_X\leq B \|b\|_{X_d}. \]
Now assume that $g\in X$ and $f^*\in X^*$. Then we have
\begin{eqnarray*}
(Kg)(f^*)&=&g(K^*f^*)=[K^*f^*, g^*]_*=[g, (K^*f^*)^*]\\
&=&(K^*f^*)(g)=(\sum_j[f_j, f]g_j^*)(g)=\sum_j[f_j, f][g, g_j]\\
&=&\sum_j[g, g_j][f^*, f_j^*]_*=\sum_j[g, g_j]f_j(f^*).
\end{eqnarray*}
Therefore $Kf=\sum_ja_jf_j$, when $a_f=\{a_j\}=\{[f, g_j]\}$. Note that, since $\{g_j^*\}$ is an $X_d$-Bessel sequence for $X$, there exists $C>0$ such that $\|a_f\|\leq C\|f\|$.
\end{proof}
Now we present an example in order to describe our work.
\begin{example}
Consider the space $X:=\ell^{\frac{3}{2}}(\mathbb{N}_3)$ with the semi-inner product
\[ [g, h]:=\|g\|^{\frac{1}{2}}\sum_{j=1}^3g_j\overline{h_j}|h_j|^{-\frac{1}{2}}, \quad g, h\in X, \]
and the following sequence in $X$:
\[ f_1=e_1, f_2=e_2, f_3=0. \]
We can easily show that for BK-space $X_d^*:=\ell^3(\mathbb N)$, $\{f_1^*=e_1, f_2^*=e_2, f_3^*=0\}$ is not an $X_d^*$-frame for $X^*$. Now we define a bounded linear operator $K^*:X^*\to X^*$ as follows:
\[ K^*e_1=e_1, K^*e_2=e_2, K^*e_3=0. \]
We show that $\{f_1^*, f_2^*, f_3^*\}$ is an $X_d^*$-$K$-frame for $X^*$. In so doing, we have
\[ \|K^*f^*\|_{\ell^3(\mathbb N)}=\|c_1e_1+c_2e_2\|_{\ell^3(\mathbb N)}=(|c_1|^3+|c_2|^3)^{\frac{1}{3}}\leq \|\{[f_1, f], [f_2, f], [f_3, f]\}\|_{\ell^3(\mathbb N)}, \]
where $f^*=c_1e_1+c_2e_2+c_3e_3$, for some $c_1, c_2, c_3\in \mathbb{C}$.
\end{example}
A characterization of an $X_d^*$-$K$-frame is given below.
\begin{theorem}
Let $X_d$ be a CB-space. Then $\{f_j\}$ is an $X_d^*$-$K$-frame if and only if there exists a bounded linear operator $U:X_d \to X$ such that $Ue_j=f_j$ and $R(K)\subseteq R(U)$.
\end{theorem}
\begin{proof}
Since $\{f_j\}$ is an $X_d^*$-$K$-frame, we can define $U:X_d \to X$ by $U(c)=\sum_jc_jf_j$, $c=\{c_j\}\in X_d$. Therefore $Ue_j=f_j$ and $U$ is bounded, i.e. $\|U\|\leq B$, where $B$ is the upper $X_d^*$-$K$-frame bound of $\{f_j\}$. By the similar way of Lemma \ref{bess} $\{U^*(f^*)(e_j)\}=\{[f_j, f]\}$ is identified with $U^*f^*$ for every $f^*\in X^*$. Therefore
\[ \|\{[f_j, f]\}\|_{X_d^*}=\|U^*f^*\|\leq B \|f^*\|_{X^*}. \]
Now by $\|K^*f^*\|\leq \|\{[f_j, f]\}\|_{X_d^*}=\|U^*f^*\|$ and Lemma \ref{maj} we have $R(K)\subseteq R(U)$.

Conversely, by the similar way $\{U^*(f^*)(e_j)\}=\{[f_j, f]\}$ is identified with $U^*f^*$, $f^*\in X^*$. Therefore
\[ \|\{[f_j, f]\}\|_{X_d^*}=\|U^*f^*\|\leq \|U\| \|f^*\|_{X^*}. \]
Since $R(K)\subseteq R(U)$ then by Lemma \ref{maj} there exists $A>0$ such that $A\|K^*f^*\|\leq \|U^*f^*\|=\|U^*(f^*)(e_j)\|=\|\{[f_j, f]\}\|$. It means, $\{f_j\}$ is an $X_d^*$-$K$-frame.
\end{proof}
In the following part some results are given.
\begin{proposition}\label{prop k}
Suppose that $\{f_j\}$ is an $X_d^*$-frame for $X^*$ and $Q\in B(X)$. Then $\{Qf_j\}$ is an $X_d^*$-frame for $X^*$ if and only if $Q^*$ is bounded below.
\end{proposition}
\begin{proof}
Let $f\in X$ then we have
\begin{equation}\label{bb}
 [Qf_j, f]=f^*(Qf_j)=Q^*f^*(f_j)=[f_j, (Q^*f^*)^*].
\end{equation}
Let $\{f_j\}$ be an $X_d^*$-frame for $X^*$ with upper $X^*_d$-frame bound $B$ and $\{Qf_j\}$ be an $X_d^*$-frame for $X^*$ with lower $X_d^*$-frame bound $C$. By (\ref{bb}) we have
\[ C\|f^*\|\leq \|\{[Qf_j, f]\}\|=\|\{[f_j, (Q^*f^*)^*]\}\|\leq B\|Q^*f^*\|,\]
therefore $\|Q^*f^*\|\geq \frac{C}{B}\|f^*\|$, i.e. $Q^*$ is a bounded below operator.\\
Now let $\{f_j\}$ be an $X_d^*$-frame for $X^*$ with $X^*_d$-frame bounds $A$ and $B$. Then
 $\|\{[f_j, (Q^*f^*)^*]\}\|\leq B\|Q^*f^*\|\leq B\|Q\| \|f^*\|.$
On the other hand
$\|\{[f_j, (Q^*f^*)^*]\}\|\geq A\|Q^*f^*\|,$ since $Q^*$ is bounded below, there exists $\gamma>0$ such that $\|Q^*f^*\|\geq \gamma \|f^*\|$. Therefore
$\|\{[f_j, (Q^*f^*)^*]\}\|\geq A\gamma \|f^*\|$. Hence by (\ref{bb}) $\{Qf_j\}$ is an $X_d^*$-frame for $X^*$.
\end{proof}
The following two propositions are proved by the similar way of Proposition \ref{prop k}.
\begin{proposition}
Let $K\in B(X)$ and $\{f_j\}$ be an $X_d^*$-frame for $X^*$ with $X_d^*$-frame bounds $A$ and $B$, respectively. Then $\{Kf_j\}$ is an $X_d^*$-$K$-frame for $X^*$ $X_d^*$-$K$-frame bounds $A$ and $B\|K\|$, respectively.
\end{proposition}
\begin{proposition}
Let $\{f_j\}$ be an $X_d^*$-$K$-frame for $X^*$. Then $\{f_j\}$ is an $X_d^*$-frame for $X^*$ if $K$ is a bounded below operator.
\end{proposition}
\section{Sampling in a s.i.p. RKBS}
The main result of previous section in any s.i.p. RKBSs is given in this section. First, we state some notations needed for our next result.

We mention that $X$ is \textit{\textit{uniformly Fr\'{e}chet differentiable}} if for all $f, g\in X$
\begin{eqnarray}\label{FD}
 \lim_{\lambda\in \mathbb R, \lambda\to 0}\frac {\|f+\lambda g\|_X-\|f\|_X}{\lambda}
\end{eqnarray}
exists and the limit is approached uniformly for all $f, g$ in the unit ball of $X$. If $X$ is uniformly Fr\'{e}chet differentiable, then it has a unique compatible semi-inner product \cite{G} and see also \cite{ZX}. The differentiability (\ref{FD}) of the norm is useful to derive characterization equations for the minimizer of regularized learning schemes in Banach spaces. For simplicity, we call a Banach space uniform if it is both uniformly convex and uniformly Fr\'{e}chet differentiable. Notice that its dual $X^*$ is also uniform \cite{DFC}. In this section, $X$ is called a s.i.p. RKBS on $\Omega$ if it is a uniform Banach space of functions on $\Omega$ where point evaluations are always continuous linear functionals. Also its s.i.p. reproducing kernel is denoted by $k$. For the theory of RKBSs see for instance \cite{ZX} and references therein. Most important of all, by the arguments in the proof of Theorem 9 in \cite{ZX}, there exists a unique function $G:X\times X\to \mathbb{C}$ such that $G(t, .)\in X$ for all $t\in \Omega$ and $f(t)=[f, G(t, .)]$, for all $t\in \Omega$ and $f\in X$. By virtue of the above equation, $G$ is called the s.i.p. reproducing kernel of $X$. Moreover, there holds the relationship $k(., t)=(G(t, .))^*$, $t\in \Omega$ and $f^*(t)=[k(t, .), f]$ for all $f\in X$ and $t\in \Omega$. Set $K_{\mathcal Z}:=\{G(., t_j)^*\}=\{k(t_j, .)\}$. The sampling operator $\mathcal {I}_{\mathcal Z}:X^*\to X_d^*$ is defined by $\mathcal {I}_{\mathcal Z}(f^*)=\{f^*(t_j)\}$, i.e. $\mathcal {I}_{\mathcal Z}(f^*)=\{[k(t_j, .), f]\}, t_j\in \Omega$. For more details one can see page 12 to 14 \cite{ZZ}. Now the result of main Theorem in any s.i.p. RKBS is given below.
\begin{theorem}
With the notations mentioned above, let $X_d$ be a CB-space and $X$ be a s.i.p. RKBS on $\Omega$ with the s.i.p. reproducing kernel $k$ and $K_{\mathcal Z}$ be an $X^*_d$-Bessel sequence for $X^*$ and $\overline{{\mathcal I}_{\mathcal Z}(X^*)}$ is complemented in $X_d^*$. Then the following statements are equivalent:\\
$(i)$ $K_{\mathcal Z}$ is an $X_d^*$-atomic system for $X^*$ with respect to $K$;\\
$(ii)$ $K_{\mathcal Z}$ is an $X_d^*$-$K$-frame for $X^*$, i.e. there exist $A, B>0$ such that
\[ A\|K^*f^*\|_{X^*}\leq \|\mathcal {I}_{\mathcal Z}(f^*)\|_{X_d^*}\leq B\|f^*\|_{X^*}, \quad for\quad all\quad f^*\in X^*; \]
$(iii)$ $K_{\mathcal Z}$ is an $X_d^*$-Bessel sequence for $X^*$ and there exists an $X_d$-Bessel sequence $\{g_j^*\}$ for $X$ such that for any $f\in X$ we have
\[ K^*f^*=\sum_jf^*(t_j)g_j^*=\sum_j[k(t_j, .), f]g_j^*.\]
\end{theorem}
A corollary of the previous Theorem is given below.
\begin{corollary}
With the notations mentioned above, let $X_d$ be a CB-space and $X$ be a s.i.p. RKBS on $\Omega$ with the s.i.p. reproducing kernel $k$ and $K_{\mathcal Z}$ be an $X^*_d$-Bessel sequence for $X^*$ and $\overline{{\mathcal I}_{\mathcal Z}(X^*)}$ is complemented in $X_d^*$. Then the following statements are equivalent:\\
$(i)$ $K_{\mathcal Z}$ is an $X_d^*$-frame for $X^*$;\\
$(ii)$ $K_{\mathcal Z}$ is an $X_d^*$-Bessel sequence for $X^*$ and there exists an $X_d$-Bessel sequence $\{g_j^*\}$ for $X$ such that for any $f\in X$, we have
\[ f^*=\sum_jf^*(t_j)g_j^*=\sum_j[k(t_j, .), f]g_j^*.\]
\end{corollary}
\section{Perturbations of $X_d^*$-$K$-frames by using $X_d^*$-$K$-frame operators}
Several authors have generalized the Paley--Wiener perturbation theorem to the perturbation of frames in Hilbert spaces. The most general result of these was obtained by Casazza and Christensen \cite[Theorem 2.1]{CC}. We show that the Casazza--Christensens perturbation theorem of Hilbert space frames somehow holds for $X_d^*$-$K$-frames. In order to state and prove this theorem, we have to define $X_d^*$-$K$-frame operator for $X^*$. The first thing which we will employ is a known result about pseudo-inverse of any bounded linear operator  \cite{SLC, MZN}.

Let $X$ and $Y$ be Banach spaces and $Q\in B(X, Y)$ be such that the range $R(Q)$ of $Q$ is closed in $Y$. Assume that $X$ is the topological sum of the null space $N(Q)$ of $Q$ and $N(Q)^c$, and $Y$ is the topological sum of $R(Q)$ and $R(Q)^c$, where $N(Q)^c$ and $R(Q)^c$ are closed subspaces of $X$ and $Y$, respectively. Note that $Q$ is one-to-one from $N(Q)^c$ onto $R(Q)$. Let $P_X$ be the projection of $X$ onto $N(Q)$ along $N(Q)^c$, and let $P_Y$ be the projection of $Y$ onto $R(Q)$ along $R(Q)^c$. The bounded linear operator $Q^{\dag}:Y\to X$ defined by $Q^{\dag}Qf=f$ for $f\in N(Q)^c$ and $Q^{\dag}g=0$ for $g\in R(Q)^c$ is called the pseudo-inverse of $Q$ (with respect to $P_X, P_Y$). In particular, for any $g \in R(Q)$, $QQ^{\dag}g=g$. If there exists a pseudo-inverse $Q^{\dag}$ of $Q$ such that $QQ^{\dag}f=f$, for any $f$ in $R(Q)$, namely $QQ^{\dag}{\mid}_{R(Q)}=I|_{R(Q)}$, then ${(Q^{\dag}{\mid}_{R(Q^*)})}^*Q^*=I{\mid}_{R(Q^*)}$. So we have $\|f^*\|=\|(Q^{\dag}\mid_{R(Q^*)})^*Q^*f^*\|\leq \|Q^{\dag}\|\|Q^*f^*\|$ for every $f^*\in R(Q^*)$. In the rest of this section when we use $K^{\dag}$, in any $X_d^*$-$K$-frame, we mean, under these conditions.

Now, we are going to define $X_d^*$-$K$-frame operator (see also \cite{sto}). Let $X$ be a separable Banach space and $X_d$ be a BK-space. In order to compose the operator $T:X^* \to X_d^*$ defined by $T(f^*)=\{[f_j, f]\}$ and the operator $U:X_d \to X$ defined by $U(\{c_j\})=\sum_jc_jf_j$, we use the duality mapping $\Phi_{X_d^*}:X_d^* \to X_d^{**}$, $\Phi_{X_d^*}(c^*)=\{ c^{**}\in X_d^{**}  :  c^{**}(c^*)=\|c^*\|^2=\|c^{**}\|^2\}$, in the case when it is single-valued.

For a given $\{f_j\}$, as an $X_d^*$-$K$-frame for $X^*$, the operators $T$ and $U$, defined above, call the \textit{analysis} and \textit{synthesis} operator for $\{f_j\}$. In the sequel we use these operators as they are defined above.

The duality mapping $\Phi_{X_d^*}$ needs being single-valued. In so doing, $X_d^*$ or just $X_d$ has to be uniformly convex BK-space because the bi-dual element of $c$, in any uniformly convex Banach space $X_d$, equals to $c$, i.e. $c^{**}=c$. Besides, if $X_d^{**}$ is strictly convex then $\Phi_{X_d^*}$ is single-valued \cite{D}. Note that a uniformly convex Banach space is automatically strictly convex and reflexive. In addition, a dual space of any Banach space is also uniformly convex Banach space. Therefore, $\Phi_{X_d^*}$ is a single-valued map to $X_d$. In the sequel, if there is no risk of confusion, we will omit the index and write $\Phi$ instead of $\Phi_{X_d^*}$.
\begin{definition}\label{op fr}
Let $X$ be a strictly convex separable Banach spaces, $X_d$ be a uniformly convex BK-space, and $\{f_j\}$ be an $X_d^*$-$K$-frame for $X^*$. We define the \textit{$X_d^*$-$K$-frame operator} for $\{f_j\}$ via
\[ S:=U\Phi_{X_d^*}T. \]
\end{definition}
Note that $S$ is the bounded operator from $X^*$ to $X$.
\begin{lemma}\label{analys}
Let $X_d$ be a CB-space and $\{f_j\}$ be an $X_d^*$-Bessel sequence for $X^*$, then $U^*=T$.
\end{lemma}
\begin{proof}
Let $f^*\in X^*$, then
\[ U^*(f^*)(e_j)=f^*(Ue_j)=f^*(f_j)=[f_j, f]=T(f^*)(e_j).\]
For $c=\{c_j\}\in X_d$ we have
\begin{eqnarray*}
U^*(f^*)(c)&=&U^*(f^*)(\sum_jc_je_j)=\sum_jc_jU^*(f^*)(e_j)\\
&=&\sum_jc_jT(f^*)(e_j)=T(f^*)(\sum_jc_j e_j)=T(f^*)(c),
\end{eqnarray*}
therefore, $U^*=T$.
\end{proof}
Under the assumption in Lemma \ref{analys} $S$, the $X_d^*$-$K$-frame for $\{f_j\}$, can be written $U\Phi U^*$ instead of $U\Phi T$. In the sequel we use $S$ as the $X_d^*$-$K$-frame operator for $\{f_j\}$.

The following proposition of $X_d^*$-$K$-frame operators need for our next study.
\begin{proposition}\label{kdag}
Let $X$ be a strictly convex separable Banach space, $X_d$ be a uniformly convex BK-space, and $\{f_j\}$ be an $X_d^*$-$K$-frame for $X^*$ with $X_d^*$-$K$-frame bounds $A$ and $B$, respectively. Then the following statement holds:
\begin{enumerate}
  \item [$(i)$] $A^2\|K^*f^*\|^2\leq [Sf^*, f]\leq B^2\|f^*\|^2$, for all $f^*\in X^*$.\\
Moreover, suppose that $K$ has pseudo-inverse $K^{\dag}$. For any $f^*\in R(K^*)$, the following statements hold:
  \item [$(ii)$] $A^2\|K^{\dag}\|^{-2}\|f^*\|\leq\|Sf^*\|\leq B^2\|f^*\|;$
  \item [$(iii)$]$\|Tf^*\|\leq A^{-1}\|K^{\dag}\|\|Sf^*\|.$
\end{enumerate}
\end{proposition}
\begin{proof}
\begin{enumerate}
  \item [$(i)$] Since $\{f_j\}$ is an $X_d^*$-$K$-frame for $X^*$, it is enough to show that $[Sf^*, f]=\|\{[f_j, f]\}\|_{X_d^*}^2$. For $f^*\in X^*$ we have
 \begin{eqnarray*}
[Sf^*, f]&=&[U\Phi Tf^*, f]=f^*(U\Phi Tf^*)=U^*f^*(\Phi Tf^*)\\
&=&(\{[f_j, f]\})((\{[f_j, f]\})^*)=[(\{[f_j, f]\})^*, (\{[f_j, f]\})^*]_{X_d}\\
&=&[\{[f_j, f]\}, \{[f_j, f]\}]_{X_d^*}=\|\{[f_j, f]\}\|_{X_d^*}^2.
  \end{eqnarray*}
  \item [$(ii)$] From part $(i)$ and $\|f^*\|=\|(K^{\dag}\mid_{R(K^*)})^*K^*f^*\|\leq \|K^{\dag}\|\|K^*f^*\|$ for every $f^*\in R(K^*)$, we obtain
  \begin{equation*}
  [Sf^*, f]\geq A^2\|K^*f^*\|^2\geq A^2\|K^{\dag}\|^{-2}\|f^*\|^2,
  \end{equation*}
hence
\begin{equation}\label{RS}
\|Sf^*\|\geq \frac{[Sf^*, f]}{\|f\|}\geq\frac{A^2\|K^{\dag}\|^{-2}\|f^*\|^2}{\|f\|}=A^2\|K^{\dag}\|^{-2}\|f^*\|.
\end{equation}
On the other hand
\begin{eqnarray}\label{LS}
\|Sf^*\|&=&\sup_{g^*\in X^*, \|g^*\|=1}g^*(Sf^*)=\sup_{g\in X, \|g\|=1}[Sf^*, g]\nonumber\\
&=&\sup_{g\in X, \|g\|=1}[\sum_jd_jf_j, g]=\sup_{g\in X, \|g\|=1}\sum_jd_j[f_j, g]\nonumber\\
&\leq&\sup_{g\in X, \|g\|=1} \|\{d_j\}\|\|\{[f_j, g]\}\|\leq \sup_{g\in X, \|g\|=1} B\|g\|\|\{d_j\}\|\nonumber\\
&\leq& B\|\{[f_j, f]\}\| \leq B^2\|f^*\|,
\end{eqnarray}
where $\{d_j\}:=\{[f_j, f]\}^*$, for all $j\in J$. Therefore, we get part $(ii)$ by applying (\ref{RS}) and (\ref{LS}).
\item [$(iii)$] By applying part $(i)$ and $(ii)$ we have
\begin{eqnarray*}\label{TS}
\|Tf^*\|^2&=&\|\{[f_j, f]\}\|^2=[Sf^*, f]\\
&\leq&\|Sf^*\| \|f\| \leq A^{-2}\|K^{\dag}\|^2 \|Sf^*\|^2
\end{eqnarray*}
This implies part $(iii)$.
\end{enumerate}
\end{proof}
We are now ready to state and prove our theorem about perturbation of $X_d^*$-$K$-frame. Actually, it is the generalization of \cite[Theorem 2.1]{CC}, \cite[Theorem 3.13]{XZG}, and \cite[Proposition 4.3]{ZY}.
\begin{theorem}
Assume that $X$ is a strictly convex separable Banach space, $X_d$ is a uniformly convex Banach space, and  $\{f_j\}$ is an $X_d^*$-$K$-frame for $X^*$ with $X_d^*$-frame bounds $A$ and $B$, respectively. Suppose that $\{g_j\}$ is a sequence in $X$. If there exist $\alpha, \beta, \gamma\geq0$ such that for any finite sequence $\{c_j\}\in X_d$,
\begin{equation}\label{asu}
\|\sum_jc_j(g_j-f_j)\|_X\leq \alpha \|\sum_jc_jf_j\|_X+\beta\|\sum_jc_jg_j\|_X+\gamma\|\{c_j\}\|_{X_d},
\end{equation}
is fulfilled with $\max \{\beta, \alpha+\gamma{A^{-1}}\|K^{\dag}\|\|\Phi\|\}<1$. Then $\{g_j\}$ is also an $X_d^*$-$PK$-frame for $X^*$, where $P$ is the orthogonal projection operator from $X$ to $Q(R(K^*))$, $Q:=V\Phi U^*$, $U$, $V$ are \textit{synthesis} operators for $\{f_j\}$ and $\{g_j\}$, respectively.
\end{theorem}
\begin{proof}
Since $\{f_j\}$ is an $X_d^*$-$K$-frame for $X^*$, we can define operators $U:X_d \to X$ by
\[ U(\{c_j\})=\sum_jc_jf_j,\]
Furthermore, $\|U\|\leq B$. Suppose that condition (\ref{asu}) holds for any finite sequence $\{c_j\}$. Then for each $c=\{c_j\}\in X_d$ we have that
\[ \|\sum_jc_jg_j\|\leq \frac{1+\alpha}{1-\beta}\|\sum_jc_jf_j\|+\frac{\gamma}{1-\beta}\|\{c_j\}\|.\]
So
\[ \|\sum_{j\in J_m\ J_n}c_jg_j\|\leq \frac{(1+\alpha) B+\gamma}{1-\beta}\|\sum_{j\in J_m\ J_n}c_je_j\|.\]
When $m>n$ tend to $\infty$, $\sum_{j\in J_m\ J_n}c_je_j$ tends to zero. Thus, the series $\sum_jc_jg_j$ converges for any $\{c_j\}\in X_d$ and
\[ \|\sum_jc_jg_j\|\leq \frac{(1+\alpha) B+\gamma}{1-\beta}\|\sum_jc_je_j\|.\]
By Lemma \ref{bess}, $\{g_j\}$ is an $X_d^*$-Bessel sequence for $X^*$ with bound $\frac{(1+\alpha) B+\gamma}{1-\beta}$.

Now we show that $\{g_j\}$ has a lower $X_d^*$-$K$-frame bound. The condition (\ref{asu}) turns to be
\[ \|Uc-Vc\|\leq \alpha\|Uc\|+\beta\|Vc\|+\gamma\|c\|, \quad c=\{c_j\}\in X_d.\]
For $c=\Phi_{X_d^*} U^*f^*\in X_d$ we have
\begin{eqnarray}\label{A}
\|U\Phi U^*f^*-V\Phi U^*f^*\|&=&\|Sf^*-V\Phi U^*f^*\|\\
&\leq& \alpha\|Sf^*\|+\beta\|V\Phi U^*f^*\|+\gamma\|\Phi U^*f^*\|,
\end{eqnarray}
Here we use Lemma \ref{analys}, i.e. $S=U\Phi_{X_d^*}U^*$. By part $(iii)$ of Proposition \ref{kdag} and (\ref{A}) we have
\begin{equation}\label{leq1}
\|Sf^*-V\Phi U^*f^*\|\leq (\alpha+\gamma{A^{-1}}\|K^{\dag}\|\|\Phi\|)\|Sf^*\|+\beta\|V\Phi U^*f^*\|.
\end{equation}
By triangular inequality, it follows from (\ref{leq1}) that
\begin{eqnarray}\label{gleq1}
\frac{1-(\alpha+\gamma{A^{-1}}\|K^{\dag}\|\|\Phi\|)}{1+\beta}\|Sf^*\| &\leq&\|V\Phi U^*f^*\|\nonumber\\
&\leq& \frac{1+\alpha+\gamma{A^{-1}}\|K^{\dag}\|\|\Phi\|}{1-\beta}\|Sf^*\|.
\end{eqnarray}
By Combination part $(ii)$ of Proposition \ref{kdag} and (\ref{gleq1}), for any $f^*\in R(K^*)$, we have
\begin{eqnarray}\label{gleq2}
&&\frac{\big(1-(\alpha+\gamma{A^{-1}}\|K^{\dag}\|\|\Phi\|)\big)A^2\|K^{\dag}\|^{-2}}{1+\beta}\|f^*\|\leq \|V\Phi U^*f^*\|\nonumber\\
&\leq& \frac{\big(1+\alpha+\gamma{A^{-1}}\|K^{\dag}\|\|\Phi\|\big)B^2}{1-\beta}\|f^*\|.
\end{eqnarray}
Next we show that $R(Q:=V\Phi U^*)$ is closed. Suppose that $\{f_n\}_{n=1}^{\infty}\subseteq R(Q)$ and $f_n \to f$ as $n \to \infty$. Then we can find $g_n\in R(K^*)$ such that $Q(g_n)=f_n$. It follows from inequality (\ref{gleq2}) that $\{g_n\}_{n=1}^{\infty}$ is a Cauchy sequence. Suppose that $g_n \to g$ as $n \to \infty$. Therefore $f_n=Qg_n \to Qg=f$ as $n\to \infty$. Which implies that $R(Q)$ is closed. From (\ref{gleq2}), $Q:R(K^*) \to R(Q)$ is invertible. By (\ref{gleq2}) we obtain that for any $f\in Q(R(K^*))$,
\begin{equation}
\|Q^{-1}(f)\|\leq \frac{1+\beta}{\big(1-(\alpha+\gamma{A^{-1}}\|K^{\dag}\|\|\Phi\|)\big)A^2\|K^{\dag}\|^{-2}}\|f\|.
\end{equation}
On the other hand, for any $f\in X$, we also have
\[Pf=QQ^{-1}Pf=V(\Phi U^*Q^{-1}Pf)=\sum_j(\Phi U^*Q^{-1}Pf)_jg_j\]
Hence for arbitrary $f^*\in X^*$, we get
\begin{eqnarray*}
\|K^*P^*f^*\|&=&\sup_{g\in X, \|g\|=1}| \|(K^*P^*f^*)(g)\|=\sup_{g\in X, \|g\|=1} \|f^*(PKg)\|\\
&=&\sup_{g\in X, \|g\|=1}|[PKg, f]|=\sup_{g\in X, \|g\|=1}|[\sum_j(\Phi U^*Q^{-1}PKg)_jg_j, f]|\\
&=&\sup_{g\in X, \|g\|=1}|\sum_j(\Phi U^*Q^{-1}PKg)_j[g_j, f]|\\
&\leq& \sup_{g\in X, \|g\|=1}\|\Phi U^*Q^{-1}PKg\|_{X_d} \|\{[g_j, f]\}\|_{X_d^*}\\
&\leq& \sup_{g\in X, \|g\|=1}\|\Phi\| \|U^*\| \|Q^{-1}PKg\| \|\{[g_j, f]\}\|_{X_d^*}\\
&\leq& \sup_{g\in X, \|g\|=1}  \frac{B \|\Phi\|(1+\beta)}{1-(\alpha+\gamma{A^{-1}}\|K^{\dag}\|\|\Phi\|)A^2\|K^{\dag}\|^{-2}}\|K\|\|g\| \|\{[g_j, f]\}\|_{X_d^*},
\end{eqnarray*}
so we obtain the lower $X_d^*$-$K$-frame bound condition,
\[\frac{1-(\alpha+\gamma{A^{-1}}\|K^{\dag}\|\|\Phi\|)A^2\|K^{\dag}\|^{-2}}{B \|\Phi\|(1+\beta)\|K\|}\|(PK)^*f^*\|\leq \|[g_j, f]\|.\]
Therefore, $\{g_j\}$ is the $X_d^*$-$PK$-frame for $X^*$.
\end{proof}

\end{document}